\documentclass[11pt,a4paper]{article}
\usepackage[T1]{fontenc}
\usepackage[utf8]{inputenc}
\usepackage{lmodern}
\usepackage{amsmath,amssymb,amsthm}
\usepackage[margin=25mm]{geometry}
\usepackage{microtype}
\usepackage{needspace}
\usepackage{xcolor}
\usepackage{hyperref}
\hypersetup{
  colorlinks=true,
  linkcolor=blue!50!black,
  citecolor=blue!50!black,
  urlcolor=blue!50!black,
  pdftitle={A d-algebra whose product does not extend to its Dedekind completion},
  pdfauthor={Mohamed Amine Ben Amor}
}
\newcommand{\N}{\mathbb N}
\newcommand{\norm}[1]{\left\lVert #1\right\rVert_\infty}
\newtheorem{theorem}{Theorem}[section]
\newtheorem{lemma}[theorem]{Lemma}
\newtheorem{proposition}[theorem]{Proposition}
\newtheorem{corollary}[theorem]{Corollary}
\theoremstyle{remark}

\numberwithin{equation}{section}
\title{A \texorpdfstring{$d$}{d}-algebra whose product does not extend
to its Dedekind completion}
\author{Mohamed Amine Ben Amor\thanks{%
Research Laboratory of Algebra, Topology, Arithmetic, and Order,
Department of Mathematics, Faculty of Mathematical, Physical and Natural Sciences of Tunis,
Tunis El Manar University, Tunisia.\\
E-mail: \href{mailto:mohamedamine.benamor@ipest.rnu.tn}{mohamedamine.benamor@ipest.rnu.tn}.}}
\date{}

\begin{document}
\maketitle
\begin{abstract}
We construct a Banach lattice \(d\)-algebra with a strong order unit whose
product has no positive associative bilinear extension to its Dedekind
completion. This gives a negative answer to a question of Wickstead and to
the \(d\)-algebra case of a question of Huijsmans. It also disproves an
extension theorem of Chil. The example shows that norm completeness and a
strong order unit do not ensure that multiplication extends to the Dedekind
completion.
\end{abstract}



\section{Introduction}

Wickstead \cite[Question 4.1, p.~811]{Wickstead2017} asked whether the
product of a Banach lattice algebra can be extended to its Dedekind
completion. We give a negative answer even for \(d\)-algebras with a strong
order unit.

All vector lattices and algebras in this paper are real. A lattice-ordered
algebra is a vector lattice with an associative bilinear product for which
the product of two positive elements is positive. It is a \(d\)-algebra if
multiplication by a positive element is a lattice homomorphism on either
side. This is equivalent to
\[
 |a*b|=|a|*|b| \qquad(a,b\in A).
\]
A \(d\)-algebra need not be commutative or have an identity element. We use
the definitions in Bernau and Huijsmans \cite{BernauHuijsmans1990} and the
vector lattice terminology in Aliprantis and Burkinshaw
\cite{AliprantisBurkinshaw2006}. We write \(a\perp b\) for
\(|a|\wedge|b|=0\).

Our main result is the following.

\begin{theorem}
There exists a Banach lattice \(d\)-algebra with a strong order unit and a
submultiplicative order-unit norm whose product has no positive associative
bilinear extension to its Dedekind completion.
\end{theorem}

The obstruction is associativity: no positive bilinear extension can be
associative. We first prove this for an algebra of bounded matrices
(Theorem~\ref{thm:noextension}). Its norm completion gives the Banach lattice
example
(Proposition~\ref{prop:banach}).

Huijsmans \cite[Section 7, p.~166]{Huijsmans1991} asked whether the product
of an Archimedean lattice-ordered algebra can be extended to its Dedekind
completion. He also asked the same question for almost \(f\)-algebras and
\(d\)-algebras. His survey was written in the spring of 1990 and published
in 1991. Buskes and van Rooij \cite{BuskesVanRooij2000} gave a positive answer
for almost \(f\)-algebras. Chil \cite[Theorem 6]{Chil2004} claimed a positive
answer for all Archimedean \(d\)-algebras.

Theorem~\ref{thm:noextension} disproves this claim and gives a negative
answer to the \(d\)-algebra part of Huijsmans's question, thirty-six years
after the survey was written.

Jaber \cite{Jaber2020} proved that the Fremlin projective tensor
product of Banach lattice algebras is again a Banach lattice algebra.
Azouzi and Jaber \cite{AzouziJaber2025} proved that a Banach lattice algebra
with a contractive approximate identity has a positive approximate identity.

The product is noncommutative, so the example does not contradict Toumi's
extension result for commutative \(d\)-algebras
\cite[Corollary 3.3]{Toumi2013}. We also show that no positive bilinear
extension in our example can satisfy the extra condition in
\cite[Definition 3.1 and Corollary 3.2]{Toumi2013}.

The construction uses ordinary limits. Toumi
\cite[Example 3.1]{Toumi2013} also uses limits of eventually constant
sequences to give an extension that loses orthosymmetry. Here we rule out
every positive associative extension. For related representation results,
see Boulabiar \cite{Boulabiar2004,Boulabiar2007}.

\section{The example}

\subsection{Repeated rows and columns}

Let \(\N=\{1,2,\ldots\}\), and let
\[
 E=\ell^\infty(\N\times\N)
\]
be the space of all bounded infinite real matrices. The order and lattice
operations are taken entry by entry. Its norm is
\[
 \norm{a}=\sup_{i,j\in\N}|a_{ij}|.
\]
Let \(e\) be the matrix with every entry equal to \(1\).

Let \(A\) be the set of matrices in \(E\) whose rows are all equal from
some row onward, and whose columns are all equal from some column onward.
More precisely, \(a\in A\) if there is an integer \(N\geq1\) such that
\begin{equation}\label{eq:tails}
 a_{ij}=a_{i,N+1}\quad(j>N),
 \qquad
 a_{ij}=a_{N+1,j}\quad(i>N).
\end{equation}
These equalities hold for all values of the other index. The first
\((N+1)\times(N+1)\) block determines the whole matrix: repeat its last row
downward and its last column to the right. The same property holds for every
larger \(N\).

For example, the following matrices satisfy \eqref{eq:tails} with \(N=1\)
and \(N=2\), respectively:
\[
 \left(\begin{array}{c|ccc}
  1&2&2&\cdots\\ \hline
  3&4&4&\cdots\\
  3&4&4&\cdots\\
  \vdots&\vdots&\vdots&\ddots
 \end{array}\right),
 \qquad
 \left(\begin{array}{cc|ccc}
  1&2&3&3&\cdots\\
  4&5&6&6&\cdots\\ \hline
  7&8&9&9&\cdots\\
  7&8&9&9&\cdots\\
  \vdots&\vdots&\vdots&\vdots&\ddots
 \end{array}\right).
\]
The lines mark where the repetition starts. In the calculations below,
we display only the first \((N+1)\times(N+1)\) block. Its last row and
column are always understood to repeat.

\Needspace{9\baselineskip}
\begin{lemma}\label{lem:lattice}
The space \(A\) is an Archimedean vector sublattice of \(E\), with strong
order unit \(e\). Its Dedekind completion is \(E\).
\end{lemma}

\begin{proof}
If \(a,b\in A\), choose one \(N\) that works for both. Their sum, scalar
multiples, maximum and minimum all have the same repetition property.
Thus \(A\) is a vector sublattice of \(E\). It is Archimedean because the
order is entrywise. Also, \(e\in A\) and \(|a|\leq\norm{a}e\), so \(e\)
is a strong order unit.

The space \(E\) is Dedekind complete. Indeed, take any nonempty set of
matrices with an upper bound in \(E\), and take the supremum at each entry.
This gives a bounded matrix: it lies between one matrix in the set and
the given upper bound.

The space \(A\) is majorizing in \(E\), since \(|x|\leq\norm{x}e\) for
every \(x\in E\). It is also order dense. To see this, take \(0<x\in E\)
and choose a positive entry \(x_{ij}\). The matrix that has this value at
\((i,j)\) and is zero elsewhere belongs to \(A\), is nonzero, and lies
below \(x\). Thus \(A\) is order dense and majorizing in the Dedekind
complete space \(E\). These two properties identify \(E\) with \(A^\delta\).
\end{proof}

\subsection{The product}

Every row and every column of a matrix in \(A\) is eventually constant.
Define
\begin{equation}\label{eq:product}
 (a*b)_{ij}
 =
 \left(\lim_{k\to\infty}a_{ik}\right)
 \left(\lim_{k\to\infty}b_{kj}\right).
\end{equation}
In words, take the limit along row \(i\) of \(a\), take the limit along
column \(j\) of \(b\), and multiply the two numbers.

If the same \(N\) works for \(a\) and \(b\), this simply says
\[
 (a*b)_{ij}=a_{i,N+1}\,b_{N+1,j}.
\]
So the finite block of the product is obtained by multiplying the last
column of the block of \(a\) by the last row of the block of \(b\).

\begin{lemma}\label{lem:algebra}
The product \eqref{eq:product} makes \(A\) an Archimedean \(d\)-algebra,
and \(e*e=e\).
\end{lemma}

\begin{proof}
Choose one \(N\) that works for \(a,b\in A\). The formula
\((a*b)_{ij}=a_{i,N+1}b_{N+1,j}\) shows that \(a*b\) has the same repetition
property. Hence \(a*b\in A\). The product is bilinear because limits are
linear. It is positive because the limits of positive rows and columns
are nonnegative.

For associativity, choose one \(N\) that works for \(a,b,c\in A\). Applying
the product rule twice gives
\[
 ((a*b)*c)_{ij}
 =a_{i,N+1}\,b_{N+1,N+1}\,c_{N+1,j}
 =(a*(b*c))_{ij}.
\]
This holds at every entry, so the product is associative.

Finally, absolute values commute with the row and column limits. Therefore
\[
 |a*b|=|a|*|b|.
\]
For a positive fixed factor, multiplication is thus a linear map that
preserves absolute values, hence a lattice homomorphism. This proves the
\(d\)-algebra property. The formula also gives \(e*e=e\).
\end{proof}

\Needspace{27\baselineskip}
\paragraph{A calculation with \(N=1\).}
Take three matrices whose \(2\times2\) blocks are
\[
 \begin{array}{c@{\qquad}c@{\qquad}c}
 a&b&c\\[3pt]
 \begin{pmatrix}1&2\\3&4\end{pmatrix}&
 \begin{pmatrix}5&6\\7&8\end{pmatrix}&
 \begin{pmatrix}9&10\\11&12\end{pmatrix}.
 \end{array}
\]
Remember that the last row and column repeat. The block of \(a*b\) is
\[
 \begin{pmatrix}2\\4\end{pmatrix}\begin{pmatrix}7&8\end{pmatrix}
 =\begin{pmatrix}14&16\\28&32\end{pmatrix}.
\]
The block of \(b*c\) is
\[
 \begin{pmatrix}6\\8\end{pmatrix}\begin{pmatrix}11&12\end{pmatrix}
 =\begin{pmatrix}66&72\\88&96\end{pmatrix}.
\]
For the two triple products, we get
\[
 \begin{aligned}
 (a*b)*c:\quad&
 \begin{pmatrix}16\\32\end{pmatrix}\begin{pmatrix}11&12\end{pmatrix}
 =\begin{pmatrix}176&192\\352&384\end{pmatrix},\\[3pt]
 a*(b*c):\quad&
 \begin{pmatrix}2\\4\end{pmatrix}\begin{pmatrix}88&96\end{pmatrix}
 =\begin{pmatrix}176&192\\352&384\end{pmatrix}.
 \end{aligned}
\]
The blocks are equal, as the general proof predicts.

\paragraph{A calculation with \(N=2\).}
Take matrices with the following \(3\times3\) blocks:
\[
 \begin{array}{c@{\qquad}c}
 a&b\\[3pt]
 \begin{pmatrix}1&2&3\\4&5&6\\7&8&9\end{pmatrix}&
 \begin{pmatrix}1&0&2\\3&4&5\\6&7&8\end{pmatrix}.
 \end{array}
\]
The block of \(a*b\) is
\[
 \begin{pmatrix}3\\6\\9\end{pmatrix}\begin{pmatrix}6&7&8\end{pmatrix}
 =\begin{pmatrix}18&21&24\\36&42&48\\54&63&72\end{pmatrix}.
\]
For example, \((a*b)_{12}=3\cdot7=21\). Since the last row and column
repeat, we also have \((a*b)_{41}=9\cdot6=54\) and
\((a*b)_{10,20}=9\cdot8=72\).

Although \(e\) is an order unit, it is not an identity element. Let \(h\)
have its first row equal to \(1\) and all other rows equal to \(0\). Then
\[
 h*e=h,\qquad e*h=0.
\]
In particular, the product is noncommutative.

\subsection{Why no associative extension exists}

Consider the two matrices
\[
 u=\begin{pmatrix}
  0&1&1&1&\cdots\\
  0&0&1&1&\cdots\\
  0&0&0&1&\cdots\\
  0&0&0&0&\cdots\\
  \vdots&\vdots&\vdots&\vdots&\ddots
 \end{pmatrix},
 \qquad
 v=\begin{pmatrix}
  1&0&0&0&\cdots\\
  1&1&0&0&\cdots\\
  1&1&1&0&\cdots\\
  1&1&1&1&\cdots\\
  \vdots&\vdots&\vdots&\vdots&\ddots
 \end{pmatrix}.
\]
Thus \(u_{ij}=1\) when \(i<j\), and \(u_{ij}=0\) otherwise; also \(v=e-u\).
They satisfy \(0\leq u,v\leq e\) and \(u\perp v\). They belong to \(E\)
but not to \(A\): their rows and columns never start repeating.

\Needspace{10\baselineskip}
\begin{theorem}\label{thm:noextension}
The product of \(A\) has no positive associative bilinear extension to
\(E=A^\delta\).
\end{theorem}

\begin{proof}
Let \(\star\) be any positive bilinear extension of \(*\) to \(E\).
We will first use positivity to prove that
\begin{equation}\label{eq:forced}
 u\star e=e,\qquad e\star v=e,\qquad e\star u=0.
\end{equation}
We will then show that these equalities rule out associativity.

\emph{First, test one row at a time.}
Fix \(n\in\N\). Let \(p\) be the matrix that is \(1\) in row \(n\) to the
right of the diagonal and is \(0\) elsewhere. This matrix belongs to \(A\)
with \(N=n\), and \(0\leq p\leq u\). The product \(p*e\) is the matrix
that is \(1\) on the whole of row \(n\) and is \(0\) elsewhere.

For example, when \(n=2\), the first \(3\times3\) blocks are
\[
 \begin{array}{c@{\qquad}c}
 p&p*e\\[3pt]
 \begin{pmatrix}0&0&0\\0&0&1\\0&0&0\end{pmatrix}&
 \begin{pmatrix}0&0&0\\1&1&1\\0&0&0\end{pmatrix}.
 \end{array}
\]
Positivity gives
\[
 p*e=p\star e\leq u\star e\leq e\star e=e.
\]
Hence every entry of row \(n\) in \(u\star e\) equals \(1\).
Since \(n\) was arbitrary, \(u\star e=e\).

\emph{Next, test one column at a time.}
Fix \(n\in\N\). Let \(q\) be the matrix that is \(1\) in column \(n\) on
and below the diagonal and is \(0\) elsewhere. It belongs to \(A\) with
\(N=n\), and \(0\leq q\leq v\). The product \(e*q\) is \(1\) on the whole
of column \(n\) and is \(0\) elsewhere.

For \(n=2\), the first \(3\times3\) blocks are
\[
 \begin{array}{c@{\qquad}c}
 q&e*q\\[3pt]
 \begin{pmatrix}0&0&0\\0&1&0\\0&1&0\end{pmatrix}&
 \begin{pmatrix}0&1&0\\0&1&0\\0&1&0\end{pmatrix}.
 \end{array}
\]
Again,
\[
 e*q=e\star q\leq e\star v\leq e\star e=e.
\]
Every entry of column \(n\) in \(e\star v\) is therefore \(1\). This holds
for every \(n\), so \(e\star v=e\). Since \(u=e-v\), bilinearity gives
\[
 e\star u=e\star(e-v)=e-e=0.
\]
This proves all three equalities in \eqref{eq:forced}.

\emph{Finally, compare the two triple products.}
We now have
\[
 (e\star u)\star e=0,\qquad
 e\star(u\star e)=e.
\]
Since \(e\neq0\), the extension is not associative. This argument applies
to every positive bilinear extension.
\end{proof}

\section{Further results}

\subsection{The condition used by Toumi}

On \(A\), multiplying first on the left by \(e\) and then on the right gives
the same result as doing these operations in the reverse order. Both maps
are positive projections. The row and column tests above show that their
positive linear extensions to \(E\) cannot commute: one order sends \(u\)
to \(0\), and the other sends \(u\) to \(e\).

We now consider the extra condition used by Toumi
\cite[Definition 3.1 and Corollary 3.2]{Toumi2013}.
A positive bilinear product \(\star\) on \(E\) is called
\emph{two-power orthosymmetric} if
\[
 x\perp y
 \quad\Longrightarrow\quad
 (a\star x)\star(y\star b)=0
 \qquad(a,b,x,y\in E_+).
\]
This condition concerns a bilinear product before associativity is known.

\Needspace{10\baselineskip}
\begin{corollary}\label{cor:twopower}
No positive bilinear extension of the product of \(A\) to \(E\) is
two-power orthosymmetric.
\end{corollary}

\begin{proof}
The proof above gives \(e\star v=e\) and \(u\star e=e\). Since \(v\perp u\),
the stated condition would require
\[
 (e\star v)\star(u\star e)=0.
\]
But its left-hand side is \(e\star e=e\neq0\).
\end{proof}

\Needspace{15\baselineskip}
\subsection{A Banach lattice example}

\noindent The next proposition gives a negative answer to Wickstead's open
Question~4.1 \cite[p.~811]{Wickstead2017}. He asks whether the product of
a Banach lattice algebra can always be extended to its Dedekind completion.
Our answer is negative even for Banach lattice \(d\)-algebras with a strong
order unit.

\begin{proposition}\label{prop:banach}
Let \(B\) be the closure of \(A\) in \(E\) for the supremum norm. The product
of \(A\) has a unique continuous bilinear extension to \(B\). With this
product, \(B\) is a Banach lattice \(d\)-algebra with strong order unit
\(e\), and its norm is submultiplicative. It is also relatively uniformly
complete. Its Dedekind completion is \(E\), but its product has no positive
associative bilinear extension to \(E\).
\end{proposition}

\begin{proof}
The supremum norm on \(A\) is exactly the order-unit norm defined by \(e\):
\begin{equation}\label{eq:unitnorm}
 \norm{a}=\inf\{t\geq0:|a|\leq t e\}\qquad(a\in A).
\end{equation}
For the general definition of this lattice norm, see
Fremlin \cite[354F--354G]{Fremlin2002}.
By the completion theorem for normed vector lattices
\cite[Theorem 4.2, p.~182]{AliprantisBurkinshaw2006}, the norm completion
\(B\) is a Banach lattice. We identify it with the norm closure of \(A\)
in \(E\). Its order and norm are the ones inherited from \(E\), and
\eqref{eq:unitnorm} holds for every \(a\in B\) as well. In particular,
\(e\) remains a strong order unit.

The norm is submultiplicative on \(A\). Indeed, positivity and \(e*e=e\)
give
\[
 |a*b|=|a|*|b|
 \leq\norm{a}\norm{b}(e*e)
 =\norm{a}\norm{b}e,
\]
so
\begin{equation}\label{eq:norm}
 \norm{a*b}\leq\norm{a}\norm{b}\qquad(a,b\in A).
\end{equation}
Thus \(A\) is a normed lattice algebra. The standard completion result
for these algebras \cite[Section 3, p.~807]{Wickstead2017} gives a continuous
positive associative bilinear product on \(B\) extending the product of
\(A\), with the same norm bound. This extension is unique by continuity
and the norm density of \(A\) in \(B\).

It remains to check the \(d\)-algebra identity. Let \(x,y\in B\), and
choose \(a_n,b_n\in A\) with \(a_n\to x\) and \(b_n\to y\) in norm.
The product and the lattice operations are norm continuous; for the latter,
see \cite[pp.~181--182]{AliprantisBurkinshaw2006}. Taking limits in the
identity on \(A\) therefore gives
\[
 |x*y|
 =\lim_{n\to\infty}|a_n*b_n|
 =\lim_{n\to\infty}\bigl(|a_n|*|b_n|\bigr)
 =|x|*|y|.
\]
Hence \(B\) is a Banach lattice \(d\)-algebra.

Relative uniform completeness also follows from the order-unit norm.
Indeed, \(\norm{x-y}\leq\varepsilon\) is equivalent to
\(|x-y|\leq\varepsilon e\), and every positive regulator in \(B\) is
bounded by a multiple of \(e\). Thus relative uniform convergence agrees
with norm convergence, and the two Cauchy conditions agree as well.
Since \(B\) is norm complete, it is relatively uniformly complete.

Finally, \(A\subset B\subset E\). Since \(A\) is order dense and majorizing
in \(E\), the same is true of \(B\). Therefore \(B^\delta=E\).
Any positive associative bilinear extension of the product of \(B\) to
\(E\) would also extend the product of \(A\), contrary to
Theorem~\ref{thm:noextension}.
\end{proof}

\section*{Acknowledgements}

The author thanks Professor Jamel Jaber for pointing out Elmiloud Chil's
paper and Wickstead's problem, and Professor Youssef Azouzi for checking
the counterexample.

\par\bigskip
\Needspace{30\baselineskip}

\end{document}